\documentclass[a4paper,12pt]{amsart}
\usepackage{comment,amsmath,amssymb,amsthm,changepage,pifont,
            graphicx,soul,xcolor,longtable,mathtools}
\usepackage[english]{babel}
\usepackage[foot]{amsaddr}
\allowdisplaybreaks

\newtheorem{theorem}{Theorem}[section]

\newtheorem{lemma}[theorem]{Lemma}

\newtheorem{proposition}[theorem]{Proposition}

\theoremstyle{remark}
\newtheorem{remark}[theorem]{Remark}
\newtheorem{definition}[theorem]{Definition}

\theoremstyle{theorem}

\newcommand{\FF}{\mathbb{F}}
\renewcommand{\AA}{\mathbb{A}}
\newcommand{\ZZ}{\mathbb{Z}}
\newcommand{\RR}{\mathbb{R}}
\newcommand{\QQ}{\mathbb{Q}}
\newcommand{\PP}{\mathbb{P}}

\newcommand{\NN}{\mathbb{N}}

\newcommand{\cB}{\mathcal{B}}
\newcommand{\cO}{\mathcal{O}}

\newcommand{\aff}{\mathrm{aff}}
\newcommand{\ns}{\mathrm{ns}}

\usepackage[colorinlistoftodos]{todonotes}
\usepackage[colorlinks=true, allcolors=blue]{hyperref}

\title{Dimension growth for affine varieties}

\author{Floris Vermeulen}
\address{Department of Mathematics, KU Leuven, Belgium}
\email{floris.vermeulen@kuleuven.be}

\date{}

\begin{document}

\begin{abstract}
We prove uniform upper bounds on the number of integral points of bounded height on affine varieties. If $X$ is an irreducible affine variety of degree $d\geq 4$ in $\AA^n$ which is not the preimage of a curve under a linear map $\AA^n\to \AA^{n-\dim X+1}$, then we prove that $X$ has at most $O_{d,n,\varepsilon}(B^{\dim X - 1 + \varepsilon})$ integral points up to height $B$. This is a strong analogue of dimension growth for projective varieties, and improves upon a theorem due to Pila, and a theorem due to Browning--Heath-Brown--Salberger. 

Our techniques follow the $p$-adic determinant method, in the spirit of Heath-Brown, but with improvements due to Salberger, Walsh, and Castryck--Cluckers--Dittmann--Nguyen. The main difficulty is to count integral points on lines on an affine surface in $\AA^3$, for which we develop point-counting results for curves in $\PP^1\times \PP^1$.

We also formulate and prove analogous results over global fields, following work by Paredes--Sasyk.  
\end{abstract}

\maketitle

\section{Introduction}

If $x\in \PP^n(\QQ)$ then we may write $x$ as $x = (x_0 : \ldots : x_n)$ where the $x_i$ are in $\ZZ$ and coprime. Recall that the height of $x$ is defined as $H(x) = \max\{|x_0|, \ldots, |x_n|\}$. If $X\subset \PP^n$, and $B$ is a positive integer, then we denote by $N(X, B)$ the number of points on $X\cap \PP^n(\QQ)$ whose height is bounded by $B$. If $X\subset \AA^n$, and $B$ is a positive integer then we denote by $N_\aff(X, B)$ the number of points in $[-B, B]^n\cap \ZZ^n$ which lie on $X$. This work is concerned with uniform bounds on the quantities $N(X,B)$ and $N_\aff(X,B)$, for algebraic varieties $X$.

Since the work of Bombieri--Pila~\cite{Bombieri-Pila}, the determinant method has been a highly influential tool to bound rational points on various geometric objects. This method is especially well-suited to proving uniform upper bounds on rational points on algebraic varieties. For example, Bombieri and Pila prove that if $C\subset \RR^2$ is an irreducible algebraic curve of degree $d$, then the number of points $(x,y)\in [-B, B]^2\cap \ZZ^2$ which lie on $C$ is bounded by $O_{d,\varepsilon}(B^{1/d+\varepsilon})$. We stress that the constant here depends only on $d$ and $\varepsilon$, and not on $C$ itself. By induction on dimension, Pila~\cite{Pila-ast-1995} showed that if $X\subset \RR^n$ is an irreducible algebraic variety of dimension $m$ and of degree $d$, then $N_\aff(X,B)$ is bounded by $O_{d,n,\varepsilon}(B^{m-1+1/d+\varepsilon})$. In the projective setting, the first results were obtained by Heath-Brown~\cite{Heath-Brown-Ann}, who developed a $p$-adic analogue of the determinant method. He proved that if $C\subset \PP^2$ is an irreducible projective curve of degree $d$, then the number of points on $C(\QQ)$ of height at most $B$ is bounded by $O_{d, \varepsilon}(B^{2/d+\varepsilon})$. Moreover, using this $p$-adic determinant method, Heath-Brown was able to prove that an irreducible projective surface of degree $d\geq 2$ has at most $O_{d,\varepsilon}(B^{2+\varepsilon})$ points of height at most $B$. In higher dimensions, this led to the \emph{dimension growth conjecture}, a non-uniform version of which was also conjectured by Serre~\cite{Serre-Mordell}, which states that if $X\subset \PP^n$ is an irreducible projective variety of degree $d\geq 2$, then 
\[
N(X,B)\ll_{d,n,\varepsilon} B^{\dim X + \varepsilon}.
\]
For $d\geq 6$, this conjecture was proven by Browning--Heath-Brown--Salberger~\cite{Brow-Heath-Salb}. Their methods are based on an analogue for affine varieties, which is well-suited to induction on dimension. By developing a global $p$-adic determinant method, Salberger~\cite{Salberger-dgc} was further able to prove the conjecture when $d\geq 4$, and via different methods in a non-uniform way when $d=3$. For $d=2$, dimension growth is also known by work of Heath-Brown~\cite{Heath-Brown-Ann}. Based on an improvement by Walsh~\cite{Walsh} of Heath-Brown's result for projective curves, Castryck--Cluckers--Dittmann--Nguyen~\cite{CCDN-dgc} were able to obtain dimension growth results which are moreover polynomial in $d$, and without $\varepsilon$. The possibility of polynomial dependence on the degree also appears in work by Walkowiak~\cite{Walkowiak} for projective curves, but with logarithmic factors and a worse exponent than in~\cite{CCDN-dgc}.

This article is concerned with a strong version of dimension growth for affine varieties. If $X\subset \AA^n$ is an affine hypersurface of degree $d$, then Pila's bound~\cite{Pila-ast-1995}
\[
N_\aff(X, B)\ll_{d,n, \varepsilon} B^{n-2+1/d+\varepsilon}
\]
is essentially optimal, as the example $x_2=x_1^d$ shows. On the other hand, if the intersection of $X$ with the hyperplane at infinity is geometrically integral, then one may remove the $1/d$ in the exponent, by work of Browning--Heath-Brown--Salberger~\cite{Brow-Heath-Salb}. We weaken this condition significantly, and our main result strengthens Pila's theorem, while generalizing the result by Browning--Heath-Brown--Salberger.

\begin{definition}
Let $X\subset \AA^n$ be an affine variety of dimension $m$. We say that $X$ is \emph{cylindrical over a curve} if there exists a $\QQ$-linear map $\pi: \AA^n\to \AA^{n-m+1}$ such that $\pi(X)$ is a curve.
\end{definition}

If $X$ is a hypersurface in $\AA^n$ defined by a polynomial $f\in \ZZ[x_1, \ldots, x_n]$, then $X$ is cylindrical over a curve if and only if there exist linear forms $\ell_1, \ell_2\in \QQ[x_1, \ldots, x_n]$ and a polynomial $g\in \QQ[y_1, y_2]$ such that $f(x) = g(\ell_1(x), \ell_2(x))$. Geometrically, one may reformulate $X$ being cylindrical over a curve as $X$ being a cone whose cone point is an $(n-2)$-dimensional linear space at infinity.

Our main result is then as follows.

\begin{theorem}[Affine Dimension Growth]\label{thm:main.thm}
Let $X\subset \AA^n$ be an irreducible variety of dimension $m$ and of degree $d$ defined over $\QQ$ which is not cylindrical over a curve. Then for any positive integer $B$
\begin{align*}
N_{\aff}(X,B)\ll_{d,n, \varepsilon} B^{m-1+\varepsilon}, &\quad \text{ if } d\geq 4, \\
N_{\aff}(X,B)\ll_{n, \varepsilon} B^{m-2+\frac{2}{\sqrt{3}}+\varepsilon}, &\quad \text{ if } d = 3.
\end{align*}
\end{theorem}


For $d\geq 4$ this bound is essentially optimal. Indeed, as soon as $X$ contains a linear space of maximal dimension defined over $\QQ$, we have that $N(X,B)\gg B^{m-1}$. However, it seems reasonable to expect that one may remove the factor $B^\varepsilon$, at the cost of a logarithmic factor. Also, the dependence on $d$ in these bounds is likely polynomial, as in~\cite{CCDN-dgc}. In fact, our methods almost give such improvements, and there is only one very specific case where a factor $B^\varepsilon$ and the non-explicit dependence on $d$ appear. In~\cite{CDHNV}, we prove a version of this result which is polynomial in $d$, but with some technical assumptions on $f_d$. Let us mention that Salberger~\cite{Salb.upcoming} has recently also obtained a weaker version of this theorem, namely with the extra condition that the intersection of $X$ with the hyperplane at infinity has an absolutely irreducible component of degree at least $2$. Salberger also studies the case $d=2$ via different methods based on work by Heath-Brown~\cite{Heath-Brown-cubicsurf}.

Of course, if $X$ is cylindrical over a curve, then bounding $N_\aff(X, B)$ reduces to counting integral points on a curve. Indeed, if $\pi: \AA^n\to \AA^{n-m+1}$ is a linear map for which $\pi(X)$ is a curve, then
\[
N_{\aff}(X,B) \ll B^{m-1} N_{\aff}(\pi(X), B).
\]

In positive characteristic, analogues of the Bombieri--Pila bound were obtained by Sedunova~\cite{Sedunova}, and Cluckers--Forey--Loeser~\cite{CFL} in large characteristic. In higher dimensions, the first results were obtained by the author~\cite{Vermeulen:p}, who proved dimension growth for hypersurfaces over $\FF_q(t)$ of large degree. These methods were subsequently generalized to all global fields and all degree by Paredes--Sasyk~\cite{Pared-Sas}. The results of this paper also naturally adapt to the setting of global fields, in particular also in positive characteristic. In Section~\ref{sec:global.fields} we explain how to generalize Theorem~\ref{thm:main.thm} to any global field, see Theorem~\ref{thm:main.thm.global}.

\subsection*{Proof strategy} Let us briefly discuss the proof strategy of Theorem~\ref{thm:main.thm}. One first reduces to hypersurfaces via a projection argument, and then by cutting with hyperplanes to the case of an affine surface $X$ in $\AA^3$. The most difficult case is now if $X$ is ruled, that is, $X$ is the union of the lines contained in it. Indeed, if $X$ contains only finitely many lines, then Theorem~\ref{thm:main.thm} follows already from the work of Salberger~\cite{Salberger-dgc}. Denote by $X_\infty\subset \PP^2$ the intersection of $X$ with the plane $\PP^2$ at infinity of $\AA^3$. If $X_\infty$ does not contain any linear irreducible components, then an adaptation of work by Browning--Heath-Brown--Salberger~\cite{Brow-Heath-Salb} gives the result also in this case, see e.g.\ the recent~\cite{Salb.upcoming}. Therefore, the most difficult case is when $X\subset \AA^3$ is an irreducible ruled surface for which $X_\infty$ contains a (rational) line. To obtain a good count on integral points in this case, we prove a uniform upper bound for counting integral points on curves in $\AA^1\times \PP^1$.

\begin{theorem}[{{Theorem~\ref{thm:curves.in.A1P1}}}]
Let $f(x,t)\in \ZZ[x,t]$ be irreducible with $\deg_x f = d_1, \deg_t f = d_2$. Then for any positive integers $B_1, B_2\geq 2$, the number of $(x,t)\in \ZZ\times \QQ$ with $|x|\leq B_1, H(t)\leq B_2$ and $f(x,t) = 0$ is bounded by
\[
\ll (d_1d_2)^{\frac{9}{2}} B_1^{\frac{1}{2d_2}} B_2^{\frac{1}{d_1}}\log (B_1 B_2).
\]
\end{theorem}

We prove this result in Section~\ref{sec:curves.in.P1P1} via the determinant method, in the spirit of Heath-Brown~\cite{Heath-Brown-Ann}, but with the improvements from Walsh~\cite{Walsh} and Castryck--Cluckers--Dittmann--Nguyen~\cite{CCDN-dgc}. In Section~\ref{sec:lines.on.surfaces} we then prove Theorem~\ref{thm:main.thm} for surfaces in $\AA^3$. Finally, we finish the proof in Section~\ref{sec:dim.growth} via the aforementioned projection and slicing arguments.



\subsection*{Notation} We use the notation $a = O_c(b)$ or $a\ll_c b$ to mean that there exists a constant $\alpha > 0$ depending on $c$ such that $|a|\leq \alpha |b|$. For functions $f,g: \NN\to \RR$ we use the notation $f = o(g)$ to mean that $\lim f/g = 0$.

If $f$ is a polynomial, then we denote by $V(f)$ the variety defined by the vanishing of $f$. Depending on context, this might be an affine or a projective variety.

\subsection*{Acknowledgements} The author thanks Raf Cluckers, Yotam Hendel, Kien Nguyen, and Per Salberger for many discussions around the topics of this paper. The author is supported by F.W.O.\ Flanders (Belgium) with grant number 11F1921N.

\section{Preliminaries}

In this section we recall some results about auxiliary polynomials and on bounding rational and integral points on curves. 

For many results about irreducible polynomials or varieties, it is typically enough to work with the absolutely irreducible case. 

\begin{proposition}\label{prop:abs.irre.poly}
Let $X$ be a projective or affine variety defined over $\QQ$, and of degree $d$. If $X$ is irreducible but not absolutely irreducible, then there exists a codimension $1$ subvariety $Y\subset X$ defined over $\QQ$ of degree at most $d^2$ such that $X(\QQ)\subset Y(\QQ)$.
\end{proposition}

\begin{proof}
This is established in the first paragraph of~\cite[Sec.\,4]{Walsh}.
\end{proof}

If $f$ is a polynomial over $\ZZ$ then we denote by $||f||$ the maximum of the absolute values of the coefficients of $f$. The following result by Heath-Brown allows us to replace factors $\log ||f||$ and $||f||^\varepsilon$ by factors of the form $O_{d,n}(\log B)$ and $O_{d,n,\varepsilon}(B^\varepsilon)$.

\begin{proposition}\label{prop:replace.height}
Let $f\in \ZZ[x_1, \ldots, x_n]$ be primitive, and absolutely irreducible of degree $d$, and let $B$ be a positive integer. Then either $||f||$ is bounded by $O_n\left(B^{d\binom{d+n}{n}}\right)$, or there exists a polynomial $g\in \ZZ[x_1, \ldots, x_n]$ of degree at most $d$, not divisible by $f$ and vanishing at all integral points of $V(f)$ of height at most $B$.
\end{proposition}

\begin{proof}
This is~\cite[Thm.\,4]{Heath-Brown-Ann}.
\end{proof}

We recall the following trivial bound on counting integral points on affine varieties, which is also called the Schwartz--Zippel bound.

\begin{proposition}\label{prop:schwarz.zippel}
Let $X\subset \AA^n$ be a (possibly reducible) affine variety of degree $d$ and dimension $m$. Then
\[
N_\aff (X, B)\leq d(2B+1)^m.
\]
\end{proposition}

\begin{proof}
See e.g.\ ~\cite[Sec.\, 2]{Heath-Brown-Ann} or~\cite[Thm.\,1]{Browning-HB}.
\end{proof}

%

\section{Rational points on curves in $\PP^1\times \PP^1$}\label{sec:curves.in.P1P1}

The goal of this section is to establish uniform upper bounds on counting rational points on curves in $\PP^1\times \PP^1$ and points in $\ZZ\times \QQ$ on curves in $\AA^1\times \PP^1$. The main results are Theorems~\ref{thm:curves.in.P1P1} and~\ref{thm:curves.in.A1P1}.


\subsection{Geometry of $\PP^1\times \PP^1$}

We recall some geometry of $\PP^1\times \PP^1$. We work over any field $k$. The Picard group of $\PP^1\times \PP^1$ is freely generated by the classes $L_1 = [\PP^1\times \{0\}]$ and $L_2 = [\{0\}\times \PP^1]$ with intersection product
\[
L_1^2 = L_2^2 = 0, \quad L_1L_2 = 1.
\]
Hence any curve $C$ in $\PP^1\times \PP^1$ is linearly equivalent to $d_1L_1 + d_2L_2$ for some integers $d_1, d_2\geq 0$. We call $(d_1, d_2)$ the \emph{bidegree} of $C$. Equivalently, $d_1$ and $d_2$ are the degrees of the maps $C\to \PP^1$ induced by the two coordinate projections $\PP^1\times \PP^1\to \PP^1$. If one of $d_1$ or $d_2$ is zero, say $d_2 = 0$, then $C$ is a union of lines of the form $\PP^1\times \{a_i\}$. Therefore, we will always assume that both $d_1$ and $d_2$ are non-zero. If $C$ is geometrically irreducible, then it is of arithmetic genus $(d_1-1)(d_2-1)$, e.g.\ by the adjunction formula. We define $|d| = d_1d_2$.  If $f\in k[X, Y, U, V]$ is homogeneous in $X, Y$ of degree $d_1$, and homogeneous in $U,V$ of degree $d_2$, then $f=0$ defines a curve in $\PP^1\times \PP^1$ of bidegree $(d_1, d_2)$. 

Now assume that $k=\QQ$, or more generally any global field as explained in Section~\ref{sec:global.fields}. If $P = ((x:y), (u:v))\in \PP^1(\QQ)\times \PP^1(\QQ)$ then we define the \emph{biheight of $P$} to be
\[
H(P) = (H(x:y), H(u:v)) \in \RR^2.
\]
If $C$ is a curve in $\PP^1\times \PP^1$ defined over $\QQ$, and if $B_1, B_2$ are positive integers, then we define the counting function
\[
N(C; B_1, B_2) = \# \{((x:y), (u:v))\in C(\QQ)\mid H(x:y)\leq B_1, H(u:v)\leq B_2 \}.
\]
If $f\in \ZZ[X,Y,U,V]$ is a bihomogeneous polynomial as above, then we also write $N(f; B_1, B_2)$ to mean $N(V(f); B_1, B_2)$. Note that by dehomogenizing, bounding $N(C; B_1, B_2)$ is essentially equivalent to counting rational points on an affine curve in $\AA^2$. Recall that if $f$ is a polynomial over $\ZZ$, then $||f||$ denotes the maximum of the absolute values of the coefficients of $f$.

\subsection{Main theorems}

The goal of this section is to prove two uniform upper bounds for counting rational points on curves in $\PP^1\times \PP^1$. The first result is about bounding $N(C; B_1, B_2)$. Recall that if $d = (d_1, d_2)$ then $|d| = d_1d_2$.

\begin{theorem}[Curves in $\PP^1\times \PP^1$]\label{thm:curves.in.P1P1}
Let $f\in \ZZ[X,Y,U,V]$ be primitive irreducible and bihomogeneous of bidegree $d=(d_1, d_2)$. For any positive integers $B_1, B_2$ we have that
\begin{align*}
N(f; B_1, B_2)&\ll |d|^{\frac{7}{2}} B_1^{\frac{1}{d_2}} B_2^{\frac{1}{d_1}} \frac{b(f)}{||f||^{\frac{1}{2|d|}}} + |d|^3 \log \left(B_1^{\frac{1}{d_2}}B_2^{\frac{1}{d_1}}\right) + |d|^{\frac{7}{2}}.
\end{align*}
\end{theorem}

The quantity $b(f)$ is related to absolute irreducibility of $f$ modulo primes, and will be defined below. It always satisfies 
\[
b(f)\ll  \max\{|d|^{-1}\log ||f||, 1\}^2.
\]
In particular, in the situation of the above theorem we have that
\[
N(f; B_1, B_2)\ll |d|^{\frac{7}{2}}B_1^{\frac{1}{d_2}}B_2^{\frac{1}{d_1}}.
\]
The second result is a mixed point counting where we bound points on $C$ of the form $((x:1), (u:v))$ with $x,u,v\in \ZZ$ and $|x|\leq B_1, H(u:v)\leq B_2$.

\begin{theorem}[Curves in $\AA^1\times \PP^1$]\label{thm:curves.in.A1P1}
Let $f\in \ZZ[X,U]$ be primitive irreducible of bidegree $d=(d_1, d_2)$. Then the number of $(x,t)\in \ZZ\times \QQ$ with $|x|\leq B_1, H(t)\leq B_2$ and such that $f(x,t)=0$ is bounded by 
\begin{align*}
\ll |d|^{\frac{9}{2}}B_1^{\frac{1}{2d_2}}B_2^{\frac{1}{d_1}}\log(B_1 B_2).
\end{align*}
\end{theorem}


To prove these results, we follow Salberger's global determinant method~\cite{Salberger-dgc}, while incorporating Walsh'~\cite{Walsh} improvement to keep track of the height of $f$. We also crucially need our upper bounds to depend polynomially on the degree, for which we follow the treatment of Castryck, Cluckers, Dittmann and Nguyen~\cite{CCDN-dgc}. This polynomial dependence on the degree will be needed in the next section to apply a technique due to Walkowiak~\cite{Walkowiak}. As an alternative to Theorem~\ref{thm:curves.in.P1P1}, one can instead try to prove an analogue of a result due to Broberg~\cite[Cor.\,1]{Broberg} for counting points in lop-sided boxes on projective curves, but with moreover polynomial dependence on the degree. Such a result would also allow one to prove Theorem~\ref{thm:main.thm}.  

The rest of this section is devoted to a proof of Theorems~\ref{thm:curves.in.P1P1} and~\ref{thm:curves.in.A1P1}.

\subsection{A determinant estimate}

We have the following determinant estimate by Salberger~\cite{Salberger-dgc}, made independent of the degree by following~\cite{CCDN-dgc}.

\begin{lemma}\label{lem:salberger.determinant}
Let $f\in \ZZ[X,Y,U,V]$ be an irreducible bihomogeneous polynomial and let $C = V(f)$ in $\PP^1 \times \PP^1$. Let $p$ be a prime and let $C_p$ the reduction of $C$ modulo $p$. Let $P$ be an $\FF_p$-point of multiplicity $\mu$ on $C_p$ and let $\xi_1, \ldots, \xi_s$ be $\ZZ$-points on $C$ with reduction $P$. Let $F_1, \ldots, F_s\in \ZZ[X,Y,U,V]$ be bihomogeneous polynomials over $\ZZ$. Then $\det(F_i(\xi_j))_{i,j}$ is divisible by $p^e$ with
\[
e \geq \frac{1}{2 \mu} s^2 - O(s).
\]
\end{lemma}

\begin{proof}
This is essentially~\cite[Lem.\,2.5]{SalbCrelle}, as stated in~\cite[Cor.\,2.5]{CCDN-dgc} without the dependence on the degree of $f$. Indeed, note that the proof is local, and so it does not matter whether $C$ is a curve in $\PP^1\times \PP^1$, or in $\PP^2$.
\end{proof}

\begin{lemma}\label{lem:walsh.determinant}
Let $p$ be a prime number. Let $f\in \ZZ[X,Y,U,V]$ be an irreducible bihomogeneous polynomial, let $\xi_1, \ldots, \xi_s$ be $\ZZ$-points on $V(f)$, and let $F_1, \ldots, F_s\in \ZZ[X,Y,U,V]$ be bihomogeneous polynomials. Then $\det(F_i(\xi_j))_{i,j}$ is divisible by $p^e$ with 
\[
e\geq \frac{1}{2}\frac{s^2}{n_p} - O(s),
\]
where $n_p$ is the number of $\FF_p$-points on $C_p$, counted with multiplicity.
\end{lemma}

\begin{proof}
The proof is identical to~\cite[Lem.\,1.4]{Salberger-dgc}. The only difference is that the bound is independent of the degree of $f$, as follows from Lemma~\ref{lem:salberger.determinant}.
\end{proof}

Since we are dealing with curves, bounding $n_p$ is straightforward.

\begin{lemma}\label{lem:determinant.bound.one.prime}
Assume that $C_p$ is geometrically integral. Then the determinant from the previous lemma is divisible by $p^e$ with
\[
e\geq \frac{1}{2}\frac{s^2}{p + O(|d|\sqrt{p})} - O(s).
\]
\end{lemma}

\begin{proof}
The curve $C_p$ is geometrically integral and of genus at most $(d_1-1)(d_2-1)$. Hence the Weil bound gives that
\[
n_p\leq p + 1 + 2(d_1-1)(d_2-1)\sqrt{p} \leq p + O(|d|\sqrt{p}).
\]
Now use the previous lemma.
\end{proof}

If $f\in \ZZ[X,Y,U,V]$ is a polynomial of bidegree $(d_1, d_2)$ then we define $b(f) = 0$ if $f$ is not absolutely irreducible, and 
\[
b(f) = \prod_p \exp\left( \frac{\log p}{p}\right),
\]
where the product is over all primes $p> |d|^2$ such that $f\bmod p$ is not absolutely irreducible. Note that our definition differs slightly from~\cite{CCDN-dgc}, the reason being that we are only interested in counting points on curves, for which the Weil bound is much stronger than other higher-dimensional effective bounds such as those by Cafure--Matera~\cite{CafureMatera}.

\begin{lemma}\label{lem:bound.b(f)}
Let $f\in \ZZ[X,Y,U,V]$ be primitive, and absolutely irreducible of bidegree $d=(d_1, d_2)$. Then
\[
b(f) \ll \max\left\{ \frac{\log ||f||}{|d|}, 1 \right\}^2.
\]
\end{lemma}

\begin{remark}
One can likely remove the exponent $2$ in this bound, by using some strong form of effective Noether polynomials in $\PP^1\times \PP^1$ as in~\cite{KALTOFEN1995, RuppertCrelle}. We will not need such an improvement however.
\end{remark}

\begin{proof}
Let $g_1\in \ZZ[X,U,Z]$ be the homogenization of $f(X,1,U,1)$ with respect to one variable $Z$, and similarly let $g_2\in \ZZ[Y,V,Z]$ be the homogenization of $f(1,Y,1,V)$ with respect to $Z$. Then both $g_1$ and $g_2$ define curves $C_1, C_2$ in $\PP^2$ which are birational to $V(f)\subset \PP^1\times \PP^1$ of degrees $e_1, e_2$ with $\max\{d_1, d_2\}\leq e_1, e_2\leq d_1+d_2$. Note that $||g_1|| = ||g_2|| = ||f||$. If $p$ is a prime such that both $g_1$ and $g_2$ are absolutely irreducible modulo $p$, then so is $f$ modulo $p$. Hence $b(f)\leq b(g_1)b(g_2)$. We now use~\cite[Cor.\,3.2.3]{CCDN-dgc} to bound $b(g_1)$ and $b(g_2)$, and conclude that $b(f)$ is bounded as stated.
\end{proof}

The following is the main determinant estimate, and is the analogue of~\cite[Thm.\,2.3]{Walsh}.

\begin{proposition}\label{prop:det.estimate}
Let $f\in \ZZ[X,Y,U,V]$ be bihomogeneous and absolutely irreducible, defining a geometrically irreducible curve $C$ in $\PP^1\times \PP^1$. Let $\xi_1, \ldots, \xi_s$ be in $C(\QQ)$, let $F_{\ell i}\in \ZZ[X,Y,U,V]$ be bihomogeneous, $1\leq \ell\leq L, 1\leq i\leq s$ and denote by $\Delta_\ell$ the determinant of $(F_{i\ell}(\xi_j))_{i,j}$. Let $\Delta$ be the greatest common divisor of all the $\Delta_\ell$. If $\Delta$ is non-zero, then 
\[
\log |\Delta| \geq \frac{s^2}{2}(\log s - O(1) - 2\log |d| - \log b(f)).
\]
\end{proposition}

\begin{proof}
Let $P$ be the set of primes $p$ such that either $p \leq |d|^2$ or $f$ is not absolutely irreducible modulo $p$. Applying Lemma~\ref{lem:determinant.bound.one.prime} to all primes $p \leq s$ which are not in $P$ yields that
\[
\log |\Delta| \geq \frac{s^2}{2} \sum_{\substack{p \leq s \\ p\notin P}} \frac{\log p}{p + O(|d|\sqrt{p})} - O(s)\sum_{p \leq s}\log p.
\]
The last term is bounded by $O(s^2)$, so we focus on the first term. 

For this, we recall the following bounds
\[
\sum_{p\leq n} \frac{\log p}{p} = \log n + O(1), \quad \sum_{p\geq n} \frac{\log p}{p^{3/2}} = O\left(\frac{1}{\sqrt{n}}\right),
\]
the first of which is Mertens' first theorem, while the second follows from the prime number theorem. Now, using that $\frac{1}{p + O(|d|\sqrt{p})} \geq \frac{1}{p} - \frac{O(|d|)}{p^{3/2}}$, we obtain
\begin{align*}
\sum_{\substack{p \leq s \\ p\notin P}} \frac{\log p}{p + O(|d|\sqrt{p})} &\geq \sum_{p \leq s}\frac{\log p}{p} - \sum_{p\in P}\frac{\log p}{p} - O( |d|) \sum_{\substack{p \leq s \\ p\notin P}} \frac{\log p}{p^{3/2}} \\
&\geq \log s - \sum_{p\leq |d|^2} \frac{\log p}{p} - \log b(f) - O(1) - O\left( |d|\sum_{p > |d|^2} \frac{\log p}{p^{3/2}}\right) \\
&\geq \log s - 2\log |d| - \log b(f) - O(1),
\end{align*}
as desired. 
\end{proof}

\subsection{Proof of Theorem~\ref{thm:curves.in.P1P1} and Theorem~\ref{thm:curves.in.A1P1}}

To prove Theorem~\ref{thm:curves.in.P1P1} we will need a coordinate transformation to ensure that the coefficient of $X^{d_1}U^{d_2}$ of $f$ is large compared to the height of $f$. If $f\in \ZZ[X,Y,U,V]$ is bihomogeneous of bidegree $(d_1, d_2)$, denote by $c_f$ the coefficient of $X^{d_1}U^{d_2}$ of $f$. 

\begin{lemma}\label{lem:coordinate.transform}
Let $f\in \ZZ[X,Y,U,V]$ be primitive and bihomogeneous of bidegree $d=(d_1, d_2)$. Then there exists a primitive polynomial $f'\in \ZZ[X,Y,U,V]$ which is bihomogeneous of bidegree $(d_1, d_2)$ with the following properties:
\begin{enumerate}
\item $||f'||\leq d_1^{d_1} d_2^{d_2}(d_1+1)(d_2+1)||f||$, 
\item $|c_{f'}|\geq \frac{||f'||}{3^{d_1+d_2} d_1^{d_1} d_2^{d_2}(d_1+1)(d_2+1)}$, and
\item for any positive integers $B_1, B_2$, we have $N(f; B_1, B_2)\leq N(f'; d_1B_1, d_2B_2)$.
\end{enumerate}
\end{lemma}

\begin{proof}
Consider the dehomogenization $f(1, Y, 1, V)\in \ZZ[Y,V]$. By applying~\cite[Lem.\,3.4.1]{CCDN-dgc} twice, we find integers $a,b$ with $|a|\leq d_1, |b|\leq d_2$ such that $|f(1,a,1,b)|\geq ||f||/3^{d_1+d_2}$. We define
\[
f'(X,Y,U,V) = f(X, Y+aX, U, V+bU),
\]
which is indeed primitive and bihomogeneous of bidegree $(d_1, d_2)$. Using the boundedness of $|a|, |b|$ we compute that
\[
||f'||\leq d_1^{d_1}d_2^{d_2} (d_1+1)(d_2+1)||f||.
\]
By construction, the coefficient of $X^{d_1}U^{d_2}$ of $f'$ is $f(1,a,1,b)$, which is of the desired size. The last fact follows directly from the fact that $|a|\leq d_1, |b|\leq d_2$.
\end{proof}

We are ready to prove Theorem~\ref{thm:curves.in.P1P1}.

\begin{proof}[Proof of Theorem~\ref{thm:curves.in.P1P1}]
We will first prove the theorem when $c_f$ is large compared to $||f||$. Afterwards we will deduce the theorem in general by using Lemma~\ref{lem:coordinate.transform}. So assume for now that $|c_f| \geq ||f||/C_d$, where 
\begin{equation}\label{eq:Cd.definition}
C_d = 3^{d_1+d_2}d_1^{d_1} d_2^{d_2} (d_1+1)(d_2+1).
\end{equation}
Also, by Proposition~\ref{prop:abs.irre.poly}, we may assume that $f$ is absolutely irreducible.

Let $S\subset \PP^1(\QQ)\times \PP^1(\QQ)$ be the set of rational points on $C$ of biheight at most $B = (B_1, B_2)$. We will show that there exists some auxiliary polynomial $g\in \ZZ[X,Y,U,V]$ not divisible by $f$, which is bihomogeneous of bidegree $(d_1M, d_2M)$ for some integer $M$ of size
\begin{equation}\label{eq:bound.M}
M\ll |d|^{3/2} B_1^{\frac{1}{d_2}} B_2^{\frac{1}{d_1}} \frac{b(f)}{||f||^{\frac{1}{2|d|}}} + |d|\log\left(B_1^{\frac{1}{d_2}} B_2^{\frac{1}{d_1}}\right) + |d|^{3/2},
\end{equation}
and such that $g$ vanishes on all points of $C$ of biheight at most $B$. The result then follows by computing the intersection product of $C$ with $V(g)$ inside $\PP^1\times \PP^1$.

Towards a contradiction, let $M$ be a positive integer and assume that there is no auxiliary polynomial $g$ of bidegree $(d_1M, d_2M) = dM$ as above. We show that $M$ is bounded by the quantity from equation~\eqref{eq:bound.M}. For use below, we can already assume that $M\gg |d|$.

Given integers $D = (D_1, D_2)$, write $\cB[D]$ for the set of bihomogeneous monomials of bidegree $D$. So $|\cB[D]| = (D_1+1)(D_2+1)$. Let $\Xi\subset S$ be a maximal subset which is algebraically independent over bihomogeneous monomials of bidegree $dM$, in the sense that applying all monomials in $\cB[dM]$ to $\Xi$ gives $s = |\Xi|$ linearly independent vectors. Let $A$ be the $s\times r$ matrix whose rows are these vectors, where $r = |\cB[dM]| = (d_1M+1)(d_2M+1)$. Note that each entry of $A$ is bounded in absolute value by $B_1^{d_1M}B_2^{d_2M}$. By assumption, every polynomial of bidegree $dM$ vanishing on $\Xi$ is in $f\cdot \cB[d_1M - d_1, d_2M-d_2]$, so that
\begin{align*}
s = |\cB[dM]| - |\cB[d_1(M-1), d_2(M-1)]| = 2|d|(M-1) + d_1 + d_2.
\end{align*}
Let $\Delta$ be the greatest common divisor of all $s\times s$ minors of $A$. Since every polynomial in $f\cdot \cB[d_1M - d_1, d_2M - d_2]$ has a coefficient of size at least $c_f\geq ||f||/C_d$, the Bombieri--Vaaler Theorem~\cite{Bombi-Vaal} yields that
\[
\Delta \leq \sqrt{|\det(AA^T)|}\left(||f||/C_d\right)^{s-r}.
\] 
Now we bound $|\det(AA^T)|$ simply by $s!(rB_1^{2d_1M}B_2^{2d_2M})^s$ and use the determinant estimate from Proposition~\ref{prop:det.estimate} to find that
\begin{align}\label{eq:first.estimate.aux.poly}
\frac{s^2}{2}&(\log s - O(1) - 2\log |d| - \log b(f)) \\
\leq &\frac{\log s!}{2} + \frac{s}{2}\log r + sd_1M\log B_1 + sd_2M\log B_2 - (r-s)(\log||f|| - \log C_d)). \nonumber
\end{align}
By Stirling's approximation we have that $\log (s!) = s\log (s) - s + O(\log s)$, and hence we may neglect the term $\log(s!)/2$ by adjusting the $O(1)$ on the left-hand-side. Similarly, since $r = (d_1M+1)(d_2M+1)$ we may also absorb the term $\frac{s}{2}\log r$ by adjusting the $O(1)$ on the left-hand-side. Now divide by $Ms$ in equation~\eqref{eq:first.estimate.aux.poly} to get that
\begin{align*}
&\frac{s}{2M}(\log s - O(1) - 2\log |d| - \log b(f)) \\
&\leq d_1\log B_1 + d_2\log B_2 - \frac{r-s}{Ms}(\log ||f|| - \log C_d).
\end{align*}
We first focus on the left-hand-side, for which we have that
\[
\frac{s}{M} = 2|d| + O\left( \frac{|d|}{M}\right),
\]
and that
\[
\log s = \log|d| + \log M - O(1).
\]
Hence the left-hand-side becomes
\[
|d|(\log M - O(1) - \log |d| - (1+O(1/M))\log b(f)).
\]
For the right-hand-side, by using that $M\gg |d|$ we have that
\[
\frac{r-s}{Ms} = \frac{|d|M^2 + O(|d|M)}{2|d|M^2 + O(|d|M)} = \frac{1}{2} \cdot \frac{1+O(1/M)}{1+O(1/M)} = \frac{1}{2} + O\left(\frac{1}{M}\right).
\]
Putting this all together, we have that
\begin{align}\label{eq:second.estimate.aux.poly}
&|d|(\log M - O(1) - \log |d| - (1+O(1/M))\log b(f)) \\\
&\leq d_1\log B_1 + d_2\log B_2 - \frac{\log ||f||}{2} + O\left( \frac{\log ||f||}{M}\right) + \frac{\log C_d}{2} + O\left(\frac{\log C_d}{M}\right). \nonumber
\end{align}
We remove the last two terms on the right-hand-side as follows. We compute, using equation~\eqref{eq:Cd.definition}, that
\begin{align*}
\log C_d &= (d_1+d_2)2\log 3 + d_1\log d_1 + d_2\log d_2 + \log(d_1+1) + \log(d_2+1)\\ &\leq |d|\log |d| + O(|d|).
\end{align*}
So we can remove the term $\log (C_d)/2$ in equation~\eqref{eq:second.estimate.aux.poly}, at the cost of replacing the term $-\log|d|$ on the left-hand-side by $\frac{-3}{2}\log |d|$. Similarly, since we have assumed that $M\gg |d|$, we may absorb the term $O(\log (C_d)/M)$ in the $O(1)$ on the left-hand-side. We end up with 
\begin{align}\label{eq:main.estimate.aux.poly}
&|d|\left(\log M - O(1) - \frac{3}{2}\log |d| - (1+O(1/M))\log b(f)\right) \\
&\leq d_1\log B_1 + d_2\log B_2 - \frac{\log ||f||}{2} + O\left(\frac{\log ||f||}{M}\right). \nonumber
\end{align}
We now distinguish two cases according to $||f||$. Assume first that $||f||\leq B_1^{8d_1}B_2^{8d_2}$ and $M\geq \log (B_1^{1/d_2}B_2^{1/d_1})$. Then we have that
\[
\frac{\log ||f||}{M} \ll |d|,
\]
and so we can absorb the term $O(\log (||f||)/M)$ in the $O(1)$ on the left-hand-side in equation~\eqref{eq:main.estimate.aux.poly}. Also, by using Lemma~\ref{lem:bound.b(f)} we can remove the term $O(\log(b(f))/M)$ by adjusting the $O(1)$ on the left-hand-side. Rearranging, we have achieved that
\[
M \leq |d|^{3/2}B_1^{1/d_2}B_2^{1/d_1} \frac{b(f)}{||f||^{1/2|d|}},
\]
which is as desired. 

Secondly, assume that $||f||> B_1^{8d_1}B_2^{8d_2}$. We may take $M\gg 1$ such that the term $O(\log(||f||)/M)$ on the right-hand-side of equation~\eqref{eq:main.estimate.aux.poly} satisfies
\[
O\left(\frac{\log ||f||}{M}\right) \leq \frac{\log ||f||}{4}.
\]
Rearranging then yields that
\[
\log M \leq O(1) + \frac{3}{2}\log |d| + 2\log b(f) - \frac{\log ||f||}{8|d|}.
\]
We now use the fact that for any $x>1$ and any $c>0$ 
\[
\log \log x - \log (x)/c\leq \log c + O(1),
\]
see e.g.\ ~\cite[Lem.\,3.3.5]{CCDN-dgc}, together with Lemma~\ref{lem:bound.b(f)} to obtain that
\begin{align*}
2\log b(f) - \frac{\log ||f||}{8|d|} &\leq \max\{4\log \log ||f|| - 4\log |d|, 0 \} - \frac{\log ||f||}{8|d|} + O(1)\\
&\leq \max\{\log|d| - 4\log |d| + O(1), 0\} + O(1) \ll 1.
\end{align*}
Hence we have that
\[
\log M \leq O(1) + \frac{3}{2}\log |d|.
\]
This completes the proof when $|c_f|\geq ||f||/C_d$. 

For a general $f$, we apply Lemma~\ref{lem:coordinate.transform} to obtain a polynomial $f'$ which satisfies $|c_{f'}|\geq ||f'||/C_d$. Applying the result to this polynomial $f'$ gives the result for $f$. Note that because of Lemma~\ref{lem:coordinate.transform}(3) this gives an extra factor $d_1^{1/d_2}d_2^{1/d_1}\leq |d|$.
\end{proof}

For the proof of Theorem~\ref{thm:curves.in.A1P1} we need a more precise version of Theorem~\ref{thm:curves.in.P1P1} depending on the height of $||f||$.

\begin{lemma}[Addendum to Theorem~\ref{thm:curves.in.P1P1}]\label{lem:addendum.thm.P1P1}
Let $f\in \ZZ[X,Y,U,V]$ be primitive irreducible and bihomogeneous of bidegree $(d_1, d_2)$, and let $B_1, B_2$ be positive integers. Assume that $||f|| > B_1^{8d_1}$. Then we have that 
\[
N(f; B_1, B_2)\ll |d|^{\frac{7}{2}}B_2^{\frac{1}{d_1}}.
\]
\end{lemma}

\begin{proof}
Follow the proof of Theorem~\ref{thm:curves.in.P1P1} up to equation~\ref{eq:main.estimate.aux.poly}. With the same notation of that proof, and using the fact that $||f|| > B_1^{8d_1}$, we have then achieved that
\[
\log M \leq O(1) + \log (B_2)/d_1 + \frac{3}{2}\log |d| + 2\log b(f) - \frac{\log ||f||}{8|d|}.
\]
The same estimates show that $2\log b(f) - \frac{\log ||f||}{8|d|} = O(1)$, concluding the proof.
\end{proof}

%
%

To prove Theorem~\ref{thm:curves.in.A1P1} we use a technique due to Ellenberg--Venkatesh~\cite{Ellenb-Venkatesh}.

\begin{proof}[Proof of Theorem~\ref{thm:curves.in.A1P1}]
Let $F\in \ZZ[X,U,V]$ be the homogenization of $f$ with respect to $U$, so that $F$ is homogeneous in $U,V$ of degree $d_2$. Let $f_i\in \ZZ[U,V]$ be the degree $i$-part, in $X$, of $F$. For an integer $H$, we define
\[
F_H(X,Y, U,V) = f_{d_1}(U,V)X^{d_1} H^{d_1} + f_{d_1-1}(U,V)X^{d_1-1}YH^{d_1-1} + \ldots + f_0(U,V)Y^{d_1}.
\]
Note that this polynomial is bihomogeneous in $X,Y$ and $U,V$ of bidegree $(d_1, d_2)$, and that $||F_H||\geq |H|^{d_1}||f_{d_1}||$. Also, every solution $(x, (u:v))$ in $\ZZ\times \PP^1(\QQ)$ of $f=0$ gives a solution $((x:H), (u:v))$ of $F_H = 0$.

Take a prime $B'$ in the range $[B_1/2, B_1]$, which exists by Bertrand's postulate, and consider $F_{B'}$. Let us first assume that $B'\nmid f_0$, so that $F_{B'}$ is primitive. We apply Theorem~\ref{thm:curves.in.P1P1} to $F_{B'}$ to find that the number of solutions $(x,u)\in \ZZ\times \QQ$ of $f=0$ with $|x|\leq B_1, H(u)\leq B_2$ is bounded by 
\[
\ll |d|^{7/2}B_1^{1/d_2}B_2^{1/d_1}\frac{b(F_{B'})}{||F_{B'}||^{1/2|d|}} + |d|^3\log(B_1^{1/d_2}B_2^{1/d_1}) + |d|^{7/2}.
\]
Now, $b(F_{B'})$ agrees with $b(f)$ up to a factor $\exp(\log (B')/B') = O(1)$, and $||F_{B'}||\geq (B')^{d_1}||f_{d_1}||\gg (B_1^{d_1}/2^{d_1})||f_{d_1}||$, so by using~\cite[Lem.\,4.2.3]{CCDN-dgc} this quantity is bounded by
\begin{align*}
&\ll |d|^{7/2}B_1^{\frac{1}{2d_2}}B_2^{\frac{1}{d_1}}\frac{\log ||f_{d_1}|| + d_1\log B_1}{||f_{d_1}||^{\frac{1}{2|d|}}} + |d|^3\log(B_1^{\frac{1}{2d_2}}B_2^{\frac{1}{d_1}}) + d^{\frac{7}{2}}
&\ll |d|^{9/2}B_1^{\frac{1}{2d_2}}B_2^{\frac{1}{d_2}}.
\end{align*}

Now assume that $f_0$ is divisible by every prime in the range $[B_1/2, B_1]$. Then
\[
\left(\prod_{\substack{B' \text{ prime}\\ B'\in [B_1/2, B_1]}} B'\right)\mid f_0.
\]
Because $F$ is irreducible, we have $f_0\neq 0$. Taking $\log$, we obtain that
\[
\sum_{\substack{B' \text{ prime}\\ B'\in [B_1/2, B_1]}} \log B' \leq \log ||f_0||.
\]
If $||f_0||> B_1^{8d_1}$ then we use Lemma~\ref{lem:addendum.thm.P1P1} to $F_1$ to find that the desired quantity is bounded by
\[
\ll |d|^{7/2}B_2^{1/d_1}.
\]
In the other case we have that
\[
\sum_{\substack{B' \text{ prime}\\ B'\in [B_1/2, B_1]}} \log B' \ll d_1\log B_1.
\]
By the prime number theorem, the left-hand-side is of order $B_1/2 + O(B_1 / \log(B_1))$. Therefore, we see that $B_1^{1/d_2}\ll |d|\log B_1$. Applying Theorem~\ref{thm:curves.in.P1P1} to $F_1$ gives that the desired quantity is bounded by 
\[
\ll |d|^{9/2}B_2^{1/d_1}\log(B_1)\frac{b(f)}{||f||^{1/2|d|}}.
\]
Using Lemma~\ref{lem:bound.b(f)} to bound $\frac{b(f)}{||f||^{1/2|d|}}$ gives the result.
\end{proof}

\section{Lines on affine surfaces}\label{sec:lines.on.surfaces}

The main new ingredient for Theorem~\ref{thm:main.thm} is a more precise count of integral points on lines on affine surfaces, improving~\cite[Prop.\,1]{Brow-Heath-Salb} and~\cite[Prop.\,4.3.3]{CCDN-dgc}. These will typically be the main contribution to $N_{\aff}(X, B)$. The goal of this section is to give strong bounds on this count, assuming that our surface is not cylindrical over a curve. By a line we mean a line defined over $\overline{\QQ}$. 

\begin{lemma}\label{lem:lines.on.surface}
Let $X\subset \AA^3$ be an irreducible affine surface of degree $d$ which is not cylindrical over a curve. Let $I$ be a finite set of lines on $X$ and denote by $Y$ the union of all lines in $I$. Then
\[
N_{\aff}(Y, B) \ll_{d, \varepsilon} B^{1+\varepsilon} + \# I.
\]
\end{lemma}

\subsection{An auxiliary lemma}

The following lemma will be crucial for counting integral points on lines on a surface.

\begin{lemma}\label{lem:count.inverse.height}
Let $F(X,T)$ be a primitive irreducible polynomial in $\ZZ[X,T]$ of bidegree $d=(d_1, d_2)$ with $d_1, d_2\geq 1$. Then
\[
\sum_{\substack{|x|\leq B \\ F(x,t)=0 \\ x\in \ZZ, t\in \QQ }} \frac{1}{H(t)} \ll_{d,\varepsilon} ||F||^{\varepsilon} B^\varepsilon.
\]
\end{lemma} 

To prove this lemma we reason as follows. Fix a positive integer $B$ and for a positive integer $i$ let $n_i$ be the number of $t\in \QQ$ of height exactly $i$ such that the polynomial $F(X,t)$ has a root in $\ZZ$. If $(x,t)\in \ZZ\times \QQ$ is such that $F(x,t) = 0$, then the bound $|x|\leq B$ also implies a bound of the form $O_{|d|}(||F|| B^{|d|})$ on $H(t)$. Hence the sum we are trying to bound is bounded by
\[
\sum_{i=1}^{O_{|d|}(||F||B^{|d|})} \frac{n_i}{i}.
\]
Therefore, to prove Lemma~\ref{lem:count.inverse.height} we need good bounds on $n_i$. We will prove that on average $n_i\ll_{d,\varepsilon} i^\varepsilon$. The following lemma is reminiscent of Walkowiak's work on effective Hilbert irreducibility~\cite{Walkowiak}.

\begin{lemma}\label{lem:when.root}
Let $F(X,T)$ be a primitive irreducible polynomial in $\ZZ[X,T]$ of bidegree $d=(d_1, d_2)$ with $d_1\geq 2$. Then the number of $t\in \QQ$ with $H(t)\leq i$ for which $F(X,t)$ has a root in $\ZZ$ is bounded by 
\[
\ll |d|^{\frac{21}{2}}  i \log(i) \log (||F||).
\]
\end{lemma}

\begin{proof}
Write $F(X,T) = F_{d_1}(T)X^{d_1} + \ldots + F_0(T)$. The number of zeroes of $F_0(t)$ is bounded by $d_2$, and so we can always assume below that $F_0(t)\neq 0$. If $(x,t/s)\in \ZZ\times \QQ$ is a zero of $F$, where $t,s\in \ZZ$ are coprime and $H(t:s) = i$, then the rational root theorem shows that 
\[
|x| \leq (d_2+1)||F||i^{d_2}.
\]
We apply Theorem~\ref{thm:curves.in.A1P1} to $F$ with $B_1 = (d_2+1)||F||i^{d_2}$ and $B_2 = i$ to obtain that the number of $t\in \QQ$ with $H(t)\leq i$ such that $F(X,t)$ has a root in $\ZZ$ is bounded by
\[
|d|^{\frac{9}{2}} ||F||^{\frac{1}{2d_2}}i^{\frac{1}{2} + \frac{1}{d_1}}\log(||F||) \log i.
\]
If $||F|| \leq e^e$ then the result follows, since $d_1\geq 2$. So we may assume that $||F|| > e^e$. Also, if $d_2 \geq \frac{\log ||F||}{2\log \log ||F||}$ then 
\[
||F||^{\frac{1}{2d_2}} \ll \log ||F||
\]
and so the result follows too. Hence we may assume that $d_2\ll \frac{\log ||F||}{\log\log ||F||}$. We now use a technique due to Walkowiak~\cite{Walkowiak}. Take
\[
E = \left\lfloor |d| \frac{\log ||F||}{\log \log ||F||}\frac{\log i}{\log \log i}\right\rfloor + 1
\]
and define $G(X,T) = F(X+T^E, T)$. Note that $\deg_X G = d_1$ and $\deg_T G\in [d_1E, d_1E+d_2]$. If $(x,t)$ is a zero of $G$ in $\QQ^2$ then $(x+t^E, t)$ is a zero of $F$. If this zero of $F$ is in $\ZZ\times \QQ$ with $H(t)\leq i$, then the rational root theorem gives that
\[
H(x) = H(x+t^E - t^E) \leq 2H(x+t^E)H(t)^E \ll d_2 ||F||i^{d_2+E}.
\]
Consider the polynomial $G'(X,Y,T,S)$ which is the homogenization of $G(X,T)$, so that $G'$ is bihomogeneous of bidegree $(\deg_X G, \deg_T G)$. Now use Theorem~\ref{thm:curves.in.P1P1} on $G'$ but with $B_1 \approx d_2||F||i^{d_2+E}$ and $B_2 = i$ to get that the desired count is at most
\[
\ll (d_1 (d_1E+d_2))^{\frac{7}{2}} ||F||^{\frac{1}{d_1 E}} i^{\frac{d_2+E}{d_1 E} + \frac{1}{d_1}}.
\]
Since $E\gg \frac{\log ||F||}{\log \log ||F||}$ we have that $||F||^{\frac{1}{d_1 E}}\ll \log (||F||)$. For the last factor, we use $d_1\geq 2$ and $E\gg d_2\frac{\log i}{\log \log i}$ to conclude that
\[
i^{\frac{d_2+E}{d_1E}+\frac{1}{d_1}}\ll i^{1+\frac{d_2}{d_1 E}} \ll i^{1+\frac{\log i}{\log \log i}} \ll i \log(i). \qedhere
\]
\end{proof}

To prove Lemma~\ref{lem:count.inverse.height} when $d_1=1$ we proceed differently.

\begin{lemma}\label{lem:rational.function.integer}
Let $F(X,T) = F_1(T)X - F_0(T)\in \ZZ[X,T]$ be primitive irreducible of bidegree $(1,d_2)$. Then for each positive integer $i$, the number of $t\in \QQ$ of height exactly $i$ such that $F(X,t)$ has a root in $\ZZ$ is bounded by $O_{\varepsilon}(d_2(2||F||)^{d_2^2\varepsilon}i^{d_2\varepsilon})$.
\end{lemma}

\begin{remark}
This is the only place in the proof of Theorem~\ref{thm:main.thm} where the factor $B^{\varepsilon}$ and the non-explicit dependence on $d$ come from. Although this result does not in general hold with a uniform bound of the form $d^{O(1)}\log ||F|| \log i$, as examples with $d_2=1$ illustrate, one may still hope for such a bound when $d_2\geq 2$. This would be sufficient to prove Theorem~\ref{thm:main.thm} with polynomial dependence on $d$, and with factors $\log(B)$ instead of $B^\varepsilon$.
\end{remark}

\begin{proof}
Let $F_1(T,S)$ and $F_0(T,S)$ be the homogenizations of $F_1$ and $F_0$ of degree $d_2$ and define the rational function $G(T,S) = F_0(T,S)/F_1(T,S)$. We are counting the number of $(t:s)\in \PP^1(\QQ)$ of height exactly $i$ such that $G(t,s)$ is an integer (where $(t:s)$ is written in lowest terms).

Let $R\in \ZZ$ be the homogeneous resultant of $F_0$ and $F_1$. Irreducibility of $F$ implies that $F_0$ and $F_1$ are coprime, and hence $R$ is non-zero. There exist homogeneous polynomials $A, B\in \ZZ[T,S]$ such that
\[
AF_0 + BF_1 = RS^{d_2}.
\]
Therefore, for every pair of coprime $t,s\in \ZZ$ we have that $\gcd(F_0(t,s), F_1(t,s))\mid Rs^{d_2}$. Now if $G(t,s)$ is an integer then $\gcd(F_0(t,s), F_1(t,s)) = F_1(t,s)$ and so we conclude that if $G(t,s)$ is an integer then
\[
F_1(t,s)\mid Rs^{d_2}.
\]
Since the resultant $R$ may be expressed as a product of pairwise differences of roots of $F_0$ and $F_1$, and each root of $F_0$ and $F_1$ is an algebraic number whose height is bounded by $||F||$, we have that $|R|\leq (2||F||)^{d_2^2}$. Now fix a positive integer $i$. For each divisor $r\mid R i^{d_2}$, the number of solutions to
\[
F_1(T,i)= r
\]
is bounded by $d_2$, and hence so is the number of $t\in \ZZ$ coprime to $i$ for which $G(t,i)\in \ZZ$. We conclude by noting that the number of divisors of $Ri^{d_2}$ is bounded by $O_\varepsilon( |R|^{\varepsilon}i^{d_2 \varepsilon})$.
\end{proof}

\begin{proof}[Proof of Lemma~\ref{lem:count.inverse.height}]
Let $(x,t)\in \ZZ\times \QQ^\times$ with $F(x,t)=0$ and $|x|\leq B$. The rational root theorem shows that $H(t)\leq (d_1+1)||F|| B^{d_1}$. So we may switch the sum around
\[
\sum_{\substack{|x|\leq B \\ F(x,t)=0 \\ x\in \ZZ, t\in \QQ^\times }} \frac{1}{H(t)} \ll \sum_{i=1}^{(d_1+1)||F|| B^{d_1}} \frac{n_i}{i},
\]
where $n_i$ is the number of $(x,t)\in \ZZ\times \QQ$ such that $F(x,t) = 0$ and $H(t) = i$. Let us now first assume that $d_1 = \deg_X F \geq 2$. Then Lemma~\ref{lem:when.root} shows that
\[
\sum_{i=1}^k n_i \leq |d|^{15/2} k \log(k)\log (||F||).
\]
We use partial summation to see that
\begin{align*}
\sum_{i=1}^{(d_1+1)||F|| B^{d_1}} \frac{n_i}{i} &= \sum_{k=1}^{(d_1+1)||F|| B^{d_1}} \left(\sum_{i=1}^k n_i\right) \left(\frac{1}{k} - \frac{1}{k+1}\right) + \left(\sum_{i=1}^{(d_1+1)||F|| B^{d_1}} n_i\right)\frac{1}{(d_1+1)||F|| B^{d_1}} \\
&\ll |d|^{21/2} \log ||F|| \left( \sum_{k=1}^{(d_1+1)||F|| B^{d_1}} \frac{k\log(k)}{k^2} + O(1)\right) \\
&\ll |d|^{25/2}\log (||F||)^2 \log(B)^2.
\end{align*}
This proves Lemma~\ref{lem:count.inverse.height} when $d_1\geq 2$.

Now assume that $d_1 = 1$. Lemma~\ref{lem:rational.function.integer} shows that in this case 
\[
n_i\ll_\varepsilon d_2 (2 ||F||)^{d_2^2\varepsilon} i^{d_2 \varepsilon} \ll_{d_2, \varepsilon'} ||F||^{\varepsilon'} i^{\varepsilon'}.
\]
Hence we can bound the sum as
\begin{align*}
\sum_{i=1}^{(d_1+1)||F|| B^{d_1}} \frac{n_i}{i} \ll_{d_2, \varepsilon} ||F||^{\varepsilon} \sum_{i=1}^{(d_1+1)||F||B^{d_1}} \frac{1}{i^{1-\varepsilon}} \ll_{d, \varepsilon} ||F||^{\varepsilon} B^{d_1\varepsilon} \ll_{d,\varepsilon'} ||F||^{\varepsilon'} B^{\varepsilon'}.
\end{align*}
This proves Lemma~\ref{lem:count.inverse.height}.
\end{proof}

\subsection{Integral points on lines on surfaces}

In this section we prove Lemma~\ref{lem:lines.on.surface}. We need a little lemma about bounding the degree of a variety.

\begin{lemma}\label{lem:bound.degree}
Let $X\subset \PP^n$ be a variety pure of dimension $m$, cut out by equations of degree $d$. Then 
\[
\deg X\leq d^{n-m}.
\]
\end{lemma}

\begin{proof}
Let $f_1, \ldots, f_N$ be non-zero homogeneous polynomials of degree $d$ such that $X = V(f_1, \ldots, f_N)$. We inductively define a sequence of polynomials $g_1, g_2, \ldots, g_k$, with $k=n-m$ and subvarieties $V_i = V(g_1, \ldots, g_i)$ for which
\[
V_{1} \supset \ldots \supset V_k \supset X,
\]
and such that $V_i$ is of dimension $n-i$ and has degree at most $d^i$. We start by putting $g_1 = f_1$ and $V_1 = V(f_1)$. Assume that $V_i, g_1, \ldots, g_i$ have been constructed, for $i < k$. Let $\{U_{ij}\}_{j=1, \ldots, M}$ be the irreducible components of $V_i$, and note that these all have dimension exactly $n-i$. Since $X$ is of smaller dimension than $V_i$, for each $j$ we can find a point $p_{ij}$ in $U_{ij}$ which is not in $X$. Since the $p_{ij}$ are not in $X$, the system of equations
\begin{align*}
\lambda_1 f_1(p_{i1}) + \ldots + \lambda_N f_N(p_{i1}) &\neq 0, \\
\vdots \quad \qquad \qquad \\
\lambda_1 f_1(p_{iM}) + \ldots + \lambda_N f_N(p_{iM}) &\neq 0
\end{align*}
has a solution $(\lambda_1, \ldots, \lambda_N)$. Now put $g_{i+1} = \lambda_1f_1 + \ldots \lambda_N f_N$ and $V_{i+1} = V_i \cap V(g_{i+1})$. 

We conclude by noting that the dimension of $V_k$ is the dimension of $X$, and hence the degree of $X$ is bounded by the degree of $V_k$, which is in turn bounded by $d^{k} = d^{n-m}$.
\end{proof}

The following lemma shows that a surface which is not cylindrical over a curve cannot contain too many parallel lines. 

\begin{lemma}\label{lem:fin.many.lines.in.direction}
Let $X\subset \AA^3$ be an absolutely irreducible surface of degree $d$ which is not cylindrical over a curve. Then for each $v\in \PP^2$, $X$ contains at most $d^2$ lines with direction $v$.
\end{lemma}

\begin{proof}
The fact that $X$ is not cylindrical over a curve shows that the number of lines on $X$ with direction $v$ is finite. Indeed, if this were infinite then the union of all lines on $X$ whose direction is $v$ would be an algebraic subvariety of $X$ of dimension $2$, hence equal to $X$ by irreducibility. But then projecting $\AA^3\to \AA^2$ along the direction of $v$ maps $X$ to a curve, contradicting the fact that $X$ is not cylindrical over a curve.

Let $f\in \ZZ[x,y,z]$ be an irreducible polynomial for which $V(f) = X$ and for $a\in \AA^3, t\in \QQ$ consider the expansion
\[
f(a+tv) = g_d(a,v)t^d + g_{d-1}(a,v)t^{d-1} + \ldots + g_0(a,v)
\]
into powers of $t$. We consider the $g_i(a,v)$ as polynomials in the variable $a$, and look at the variety in $\AA^3$ cut out by these equations. This variety is exactly the union of all lines on $X$ with direction $v$, hence of dimension $1$ (or empty). Therefore Lemma~\ref{lem:bound.degree} bounds its degree by $d^2$.
\end{proof}

We are now ready to prove Lemma~\ref{lem:lines.on.surface}.

\begin{proof}[Proof of Lemma~\ref{lem:lines.on.surface}]
Let $f\in \ZZ[x,y,z]$ be primitive irreducible of degree $d$ such that $X = V(f)$. By Proposition~\ref{prop:abs.irre.poly} we may assume that $f$ is absolutely irreducible. Indeed, if $f$ is not absolutely irreducible then Proposition~\ref{prop:abs.irre.poly} implies that there are at most $d^2$ lines in $I$ which contain at least $2$ integral points. Denote by $f_d$ the top-degree part of $f$, and let $V\subset \PP^2$ be the curve defined by $f_d=0$. The lines $\ell$ in $I$ which contain at most one integral point contribute at most $\# I$ to $N_{\aff}(Y, B)$, so it is enough to focus on the lines $\ell$ on $X$ containing at least two integral points. If $\ell$ is such a line, then there is an $a\in \ZZ^3$ and a primitive $v\in \ZZ^3$ such that $\ell(\ZZ) = \{a+tv\mid t\in \ZZ\}$. Then $v$ is a zero of $f_d$, and $\ell$ contains $\ll \frac{B}{H(v)}$ points of height at most $B$. We argue that for each irreducible component $g$ of $f_d$, the lines on $X$ whose direction $v$ lies on $V(g)$ contribute at most $\ll_{d,\varepsilon} B^{1+\varepsilon} + \#I$ to $N_{\aff}(Y,B)$. So we fix such a $g$, put $C = V(g)$, and define $Y_g$ as the union of all lines in $I$ containing at least two integral points and whose direction $v$ satisfies $g(v)=0$. We will distinguish between the case where $\deg g \geq 2$ and $\deg g = 1$.


So assume first that $\deg g \geq 2$ and let $n_i$ be the number of rational points of $C\subset \PP^2$ of height exactly $i$. By~\cite[Thm.\,2]{CCDN-dgc} we have that
\begin{equation}\label{eq:bound.sum.n_i}
\sum_{i=1}^k n_i \ll d^4 k^{\frac{2}{\deg g}}\ll d^4 k.
\end{equation}
For each $v\in C(\QQ)$, Lemma~\ref{lem:fin.many.lines.in.direction} implies that there are at most $d^2$ lines on $X$ with direction $v$. If $\ell$ is a line on $X$ with direction $v\in \PP^2$ which contains at least $2$ distinct integral points of height at most $B$, then $H(v)\leq 2B$. Therefore
\[
N_{\mathrm{aff}}(Y_g, B)\ll \#I + d^2\sum_{\substack{v\in C(\QQ) \\ H(v)\leq 2B}} \frac{2B}{H(v)} = \#I + d^2\sum_{i=1}^{B} n_i\left( \frac{2B}{i}\right).
\]
We now use partial summation and equation~\eqref{eq:bound.sum.n_i} to bound this sum as follows
\begin{align*}
\sum_{i=1}^{2B} n_i\left( \frac{2B}{i}\right) &\ll \sum_{k=1}^{2B} \left(\sum_{i=1}^k n_i\right)\frac{2B}{k^2} + \left(\sum_{i=1}^{2B} n_i\right) \\
&\ll d^4\sum_{k=1}^{2B} \frac{B}{k} + d^4B \ll d^4 B \log B.
\end{align*}
This concludes the proof when $\deg g \geq 2$.

Assume now that $\deg g = 1$, and write $g = ax+by+cz$ with $a,b,c\in \ZZ$ coprime. We first apply a coordinate transformation of $\AA^3$ as follows. Let $H$ be the plane defined by $g = 0$ in $\AA^3$ and take a shortest basis $v_2, v_3$ of the lattice $H\cap \ZZ^3$. Then Minkowski's second theorem and~\cite[Prop.\,2.12]{lang-dioph} give the bounds
\[
||v_2|| ||v_3||\asymp \det H \ll \max\{|a|, |b|, |c|\} \ll \exp(2d) ||f||.
\] 
Put $v_1 = (a,b,c)\in \ZZ^3$. Then any integral point of $\AA^3$ of height at most $B$ may be written as $\lambda_1v_1 + \lambda_2 v_2 + \lambda_3 v_3$ where $\lambda_2, \lambda_3, \det(H)^2\lambda_1\in \ZZ$ and
\[
|\lambda_2|, |\lambda_3|, \det(H)^2 |\lambda_1| \leq B.
\]
We use $v_1 / \det(H)^2, v_2, v_3$ as our new coordinate system. This gives a new polynomial $f'\in \ZZ[x,y,z]$ and a linear factor $g'$ of $f_d'$ with the following properties:
\begin{enumerate}
\item $||f'||\ll_{d} ||f||^{O_{d}(1)}$, and
\item $g' = x$, 
\item every integral point $(x,y,z)$ on $X$ of height at most $B$ corresponds to a point $(x',y',z')$ on $V(f')$ with $|y'|, |z'|\leq B$ and $|x'|\leq \exp(4d)||f||^2 B$.
\end{enumerate}

We continue with this $f'$, and put $X' = V(f')$. Note that since $X$ is not cylindrical over a curve, neither is $X'$. Now by construction of $X'$ and the fact that $g' = x$, we have that the union of all lines on $X'$ whose direction is of the form $(0,v_2,v_3)$ corresponds exactly to $Y_g\subset X$. So it suffices to count integral points on such lines. We consider $f'$ as a polynomial in $\QQ(x)[y,z]$ and denote by $F\in \ZZ[x,y,z]$ the top degree part of this polynomial. In other words, we give $y$ and $z$ weight $1$ and $x$ weight $0$, and let $F$ be the top-degree part of $f'$ as a weighted polynomial with these new weights. We note that $\deg F\leq d$, and that $\deg_{y,z} F\geq 1$. Also, since $X'$ is not cylindrical over a curve, we have that $\deg_x F\geq 1$. If $\ell$ is a line on $X'$, with base point $a = (a_1, a_2, a_3)$ and direction $v = (0, v_2, v_3)$, then the line with base point $a' = (a_1, 0, 0)$ and direction $v$ is contained in $F=0$, and furthermore $F=0$ is the union of all such lines contained in it. So it is enough to bound the integral points of height at most $B$ contained in lines on $F=0$. For this, let $G$ be an irreducible factor of $F$, and note that we still have that $\deg_{y,z} G\geq 1$ and $\deg_x G\geq 1$. If $(a_1, 0, 0)$ is a base point for a line with direction $v = (0, v_2, v_3)$ contained in $G=0$, then $G(a_1, v_3, -v_2) = 0$. Since we are only considering lines on $X'$ which contain at least $2$ integral points, the number of integral points of height at most $B$ on such a line is bounded by $\ll \frac{B}{H(v_2:v_3)}$. Hence we obtain that the total count for $G=0$ is
\[
\sum_{\substack{|a_1|\leq \exp(4d) ||f||^2B\\ G(a_1, v_3, -v_2) = 0 \\ a_1\in \ZZ, (v_2:v_3)\in \PP^2(\QQ)}}\frac{B}{H(v_2:v_3)}.
\]
The polynomial $G(x,y,z)$ is homogeneous in $y,z$, and we denote by $G'(x,t)$ the dehomogenization where we put $z=1$. Applying Lemma~\ref{lem:count.inverse.height} to $G'$ then shows that
\[
N_{\aff}(Y_g, B)\ll_{d, \varepsilon} ||f||^{\varepsilon} ||G||^{\varepsilon} B^{1+\varepsilon} + \#I.
\]
By~\cite[Prop.\,2.12]{lang-dioph} we may assume that $||G||\ll_{d} ||F||\leq ||f'||\ll_d ||f||^{O_d(1)}$.

Now, by using Proposition~\ref{prop:replace.height} we may assume that $||f||\ll B^{O_d(1)}$, from which we conclude that
\[
N_{\aff}(Y_g,B)\ll_{d, \varepsilon} B^{1+\varepsilon} + \#I. \qedhere
\]
\end{proof}

\section{Affine dimension growth}\label{sec:dim.growth}

\subsection{Affine surfaces}

To prove Theorem~\ref{thm:main.thm}, we will induct on the dimension of our variety $X\subset \AA^n$. The base case is for surfaces $X\subset \AA^3$, and so the main result of this section is as follows.

\begin{proposition}\label{prop:counting.surface}
Let $X\subset \AA^3$ be an irreducible surface of degree $d\geq 3$ which is not cylindrical over a curve. Then
\begin{align*}
N_{\aff}(X,B)\ll_{d, \varepsilon} B^{1+\varepsilon}, &\quad \text{ if } d\neq 3, \\
N_{\aff}(X,B)\ll_{\varepsilon} B^{\frac{2}{\sqrt{3}}+\varepsilon}, &\quad \text{ if } d = 3.
\end{align*}
\end{proposition}

\begin{proof}

Let $f\in \ZZ[x,y,z]$ be an irreducible polynomial of degree $d\geq 3$ defining the surface $X\subset \AA^3$. By~\cite[Thm.\,7.2]{Salberger-dgc} there exists a collection of $O_d( B^{1/\sqrt{d}}\log (B))$ geometrically integral curves $C_i$ of degree at most $O(\sqrt{d})$ such that 
\[
N_{\aff}(X_{\ns}\setminus \bigcup_i C_i, B)\ll_{d, \varepsilon} B^{2/\sqrt{d} +\varepsilon},
\]
where $X_{\ns}$ denotes the non-singular locus of $X$. By Lemma~\ref{lem:lines.on.surface} the total contribution to $N_{\aff}(X, B)$ coming from those curves $C_i$ of degree $1$ is at most $\ll_{d,\varepsilon} B^{1+\varepsilon}$. On the other hand, by~\cite[Thm.\,3]{CCDN-dgc} the curves $C_i$ of degree $>1$ contribute at most $\ll_{d, \varepsilon} B^{1/2 + 1/\sqrt{d}+\varepsilon}$ to $N_{\aff}(X, B)$. Finally, $X\setminus X_\ns$ is a possibly reducible curve of degree at most $d^2$. Distinguishing in the same way into irreducible components of degree $1$ and higher degree shows that $X\setminus X_\ns$ has at most $\ll_{d} B$ integral points of height at most $B$. Putting all these estimates together we conclude that
\begin{align*}
N_{\aff}(X, B)  \ll_{d,\varepsilon} B^{1+\varepsilon}, \quad &\text{ if } d\geq 4, \\
N_{\aff}(X, B) \ll_{\varepsilon} B^{2/\sqrt{3}+\varepsilon}, \quad &\text{ if } d=3. 
\end{align*}
This proves the result.
\end{proof}

\begin{remark}
If one can prove Lemma~\ref{lem:lines.on.surface} with polynomial dependence on $d$ and without factors $B^\varepsilon$, then one can also obtain such a result for Proposition~\ref{prop:counting.surface}. In that case, one should follow the proof from~\cite[Prop.\,4.3.4]{CCDN-dgc}.
\end{remark}

\subsection{Affine varieties}

The goal of this section is to prove our main result, Theorem~\ref{thm:main.thm}. The idea is to first prove the result for hypersurfaces by using induction and slicing with hyperplanes, and then use a projection for the general case. Note that our induction is slightly more complicated than in earlier works since we need to preserve the fact that our variety is not cylindrical over a curve. Preserving this condition is taken care of by~\cite{CDHNV}, or see~\cite[Lem.\,9]{Salb.upcoming}. However, in~\cite{CDHNV} one needs to preserve some more conditions on e.g.\ the top degree part of the defining polynomial, which is no longer required in the current work.

\begin{proof}[Proof of Theorem~\ref{thm:main.thm} for hypersurfaces]

Let $X\subset \AA^n$ be an irreducible hypersurface of degree $d\geq 3$ defined by a primitive irreducible polynomial $f\in \ZZ[x_1, \ldots, x_n]$, where $n\geq 3$. As usual, Proposition~\ref{prop:abs.irre.poly} shows that we can assume that $X$ and $f$ are absolutely irreducible. If $n=3$, then the result follows from Proposition~\ref{prop:counting.surface}, so we assume that $n\geq 4$ and that the result holds for hypersurfaces of dimension at most $n-1$. By using Noether forms exactly as in~\cite[Lem.\,4.3.7]{CCDN-dgc}, together with~\cite[Lem.\,9]{Salb.upcoming} there exists a linear form $\ell\in \ZZ[x_1, \ldots, x_n]$ such that  
\begin{enumerate}
\item there are at most $O_{d,n}(1)$ values $b\in \QQ$ such that the intersection $X\cap \{\ell=b\}$ is either not absolutely irreducible, or it is cylindrical over a curve, and
\item the height of $\ell$ is bounded by $O_{d,n}(1)$. 
\end{enumerate}
Alternatively, one may use~\cite[Prop.\,4.14]{CDHNV} to obtain an effective such linear form $\ell$, namely of height at most $O_{n}(d^3)$ and with at most $O_{n}(d^3)$ bad values of $b$. In any case, if $b\in \QQ$ is a value for which $X\cap \{\ell=b\}$ is absolutely irreducible and not cylindrical over a curve then we use induction to count integral points on $X\cap \{\ell=b\}$, while for the other values of $b$ we simply use the Schwartz--Zippel bound from Proposition~\ref{prop:schwarz.zippel} to obtain that if $d\geq 4$ then
\begin{align*}
N_{\aff}(f, B)\leq \sum_{|b|\ll_n d^3 B} N_{\aff}(X\cap \{\ell=b\}, B) \ll_{d,n,\varepsilon} B\cdot B^{n-3+\varepsilon} + B^{n-2} \ll_{d,n,\varepsilon} B^{n-2+\varepsilon}.
\end{align*}
If $d=3$, then the same argument shows that
\[
N_{\aff}(f, B)\ll_{n,\varepsilon} B^{n-3+2/\sqrt{3}+\varepsilon}. \qedhere
\]
\end{proof}

For the general case, we need a little lemma stating that projections preserve being non-cylindrical over a curve.

\begin{lemma}\label{lem:project.linear.over.curve}
Let $X\subset \AA^n$ be a variety of dimension $m$ which is not cylindrical over a curve and let $f: \AA^n\to \AA^{m+1}$ be a $\QQ$-linear map such that $f(X)$ is a hypersurface. Then $f(X)$ is not cylindrical over a curve.
\end{lemma}

\begin{proof}
Suppose that $f(X)$ is cylindrical over a curve, we show that $X$ is cylindrical over a curve. There exists a $\QQ$-linear map $\pi: \AA^{m+1}\to \AA^2$ such that $\pi(f(X))$ is a curve. Let $h: \AA^n\to \AA^{n-m-1}$ be a complementary $\QQ$-linear map to $f$, so that $f\times g: \AA^n\to \AA^{m+1}\times \AA^{n-m-1}$ is an isomorphism. Then the map $(\pi\circ f)\times h: \AA^n\to \AA^{n-m+1}$ is $\QQ$-linear and maps $X$ to a curve, so that $X$ is cylindrical over a curve.
\end{proof}

We can now prove Theorem~\ref{thm:main.thm} in general.

\begin{proof}[Proof of Theorem~\ref{thm:main.thm}]
Let $X\subset \AA^n$ be an irreducible variety of degree $d\geq 3$ and dimension $m$. By~\cite[Prop.\,4.3.1]{CCDN-dgc} there exists a $\QQ$-linear map $\pi: \AA^n\to \AA^{m+1}$ such that $\pi(X)$ is a hypersurface of degree $d$, $\pi$ is birational onto its image, and such that
\[
N_{\aff}(X, B)\leq dN_{\aff}(\pi(X), O_{d,n}(B)).
\]
By Lemma~\ref{lem:project.linear.over.curve} the image $\pi(X)$ is not cylindrical over a curve, and so Theorem~\ref{thm:main.thm} for hypersurfaces gives that if $d\geq 4$ then
\[
N_{\aff}(X, B)\ll_{d,n,\varepsilon} B^{m-1+\varepsilon}.
\]
If $d = 3$ then we obtain that
\[
N_{\aff}(X,B)\ll_{n, \varepsilon} B^{m-2+2/\sqrt{3}+\varepsilon}. \qedhere
\]
\end{proof}

\section{Global fields}\label{sec:global.fields}

\subsection{Definitions and main result}

In this section we explain how to adapt the results of this paper to global fields, including positive characteristic fields. In positive characteristic, results about counting rational and integral points on curves are due to Sedunova~\cite{Sedunova} and later Cluckers--Forey--Loeser~\cite{CFL} in large characteristic. In higher dimensions, there are results by the author~\cite{Vermeulen:p}, Paredes--Sasyk~\cite{Pared-Sas}, and Cluckers--D\`ebes--Hendel--Nguyen--Vermeulen~\cite{CDHNV}. 

We follow the notation and set-up from~\cite{Vermeulen:p, Pared-Sas}. Let $K$ be a global field, by which we mean a finite separable extension of $\QQ$ or $\FF_q(t)$. If $K$ is an extension of $\FF_q(t)$ then we assume that $\FF_q$ is the full field of constants of $K$. Denote $k = \QQ$ or $\FF_q(t)$ and $d_K = [K:k]$. Any such global field comes with a height function $H: \PP^n(K)\to \RR_{\geq 0}$, which we normalize as in~\cite{Pared-Sas} (or see~\cite[Sec.\,4]{CDHNV}). For a projective variety $X\subset \PP^n_K$ defined over $K$, we obtain the counting function $N(X, B)$ and similarly $N_{\aff}(X,B)$ for affine varieties. We say that an affine variety $X\subset \AA^n$ of dimension $m$ is \emph{cylindrical over a curve} if there exists a $\overline{K}$-linear projection $\pi: \AA^n\to \AA^{n-m+1}$ such that $\pi(X)$ is a curve. 


The analogue of Theorem~\ref{thm:main.thm} then becomes the following. 

\begin{theorem}\label{thm:main.thm.global}
Let $K$ be a global field and let $X\subset \AA^n$ be an irreducible affine variety of dimension $m$ and degree $d$ defined over $K$ which is not cylindrical over a curve. Then for any positive integer $B$
\begin{align*}
N_{\aff}(X,B)\ll_{K,d,n, \varepsilon} B^{m-1+\varepsilon}, &\quad \text{ if } d\geq 4, \\
N_{\aff}(X,B)\ll_{K,n, \varepsilon} B^{m-2+\frac{2}{\sqrt{3}}+\varepsilon}, &\quad \text{ if } d = 3.
\end{align*}
\end{theorem}


\begin{proof}
We sketch how to adapt the proof of Theorem~\ref{thm:main.thm} to this setting. First, one may adapt the results from Section~\ref{sec:curves.in.P1P1} to any global field using similar techniques as in~\cite{Vermeulen:p} and~\cite{Pared-Sas}. In particular, one obtains analogues of Theorems~\ref{thm:curves.in.P1P1} and~\ref{thm:curves.in.A1P1} which are moreover polynomial in $|d|$, but with a worse exponent. The reason for the worse exponent is due to better bounds on degrees of Noether forms in characteristic zero, compare~\cite{RuppertCrelle} with~\cite{KALTOFEN1995}. Next one can reprove the base of the induction to count $\cO_K$-points on affine surfaces in $\AA^3$ in exactly the same fashion. In this case, one should use~\cite[Thm.\,6.9]{Pared-Sas} to replace the use of~\cite[Thm.\,7.2]{Salberger-dgc} in the proof of Proposition~\ref{prop:counting.surface}. For hypersurfaces we now proceed in the same way by cutting with well-chosen hyperplanes to preserve the irreducibility and being non-cylindrical over a curve. This is possible by~\cite[Lem.\,9]{Salb.upcoming}. Finally, if $X$ is any affine variety we first project it to a hypersurface to conclude. 

\end{proof}

\bibliographystyle{amsplain}
\bibliography{anbib}

\end{document}